\documentclass[14pt]{article}

\usepackage{amssymb,amscd,amsthm,amsmath,color}
\newcommand{\PP}{\mathbb{P}}
\newcommand{\gG}{\mathbb{G}}
\newcommand{\OO}{\mathcal{O}}

\newcommand{\FF}{\mathcal{F}}

\newcommand{\uU}{\mathcal{U}}
\newcommand{\CC}{\mathbb{C}}
\newcommand{\yY}{\mathcal{Y}}
\newcommand{\xX}{\mathcal{X}}

\newcommand{\ev}{\operatorname{ev}}

\newcommand{\codim}{\operatorname{codim}}

\newcommand{\bd}{\underline{d}}
\newcommand{\mult}{\operatorname{mult}}

\newtheorem{theorem}{Theorem}[section]

\newtheorem{proposition}[theorem]{Proposition}
\newtheorem{corollary}[theorem]{Corollary}

\newtheorem{conjecture}[theorem]{Conjecture}

\theoremstyle{remark}

\author{Eric Riedl and David Yang}
\title{Applications of a grassmannian technique in hypersurfaces}

\begin{document}
\maketitle

\abstract{In this paper we further develop a Grassmannian technique used to prove results about very general hypersurfaces and present three applications. First, we provide a short proof of the Kobayashi Conjecture given previous results on the Green-Griffiths-Lang Conjecture. Second, we characterize the dimension of the space of Chow-equivalent points on a very general hypersurface, proving the remaining cases of a conjecture of Chen, Lewis and Sheng and providing a short, alternate proof for many of the already known cases. Finally, we relate Seshadri constants of very general points to Seshadri constants of arbitrary points of very general hypersurfaces.}

\section{Introduction}
In this paper, we further develop a technique from \cite{RY} and apply it to study the Kobayashi Conjecture, $0$-cycles on hypersurfaces of general type, and Seshadri constants of very general hypersurfaces. The idea of the technique is to translate results about very general points on very general hypersurfaces to results about arbitrary points on very general hypersurfaces.

Our first application is to hyperbolicity. Recall that a complex variety is Brody hyperbolic if it admits no holomorphic maps from $\CC$. 

\begin{conjecture}[Kobayashi Conjecture]
\label{conj-Kobayashi}
A very general hypersurface $X$ of degree $d$ in $\PP^n$ is Brody hyperbolic if $d$ is sufficiently large. Moreover, the complement $\PP^n \setminus X$ is also Brody hyperbolic for large enough $d$.
\end{conjecture}

First conjectured in 1970 \cite{K}, the Kobayashi Conjecture has been the subject of intense study, especially in recent years \cite{Siu, Deng, Brotbeck, D2}. The suspected optimal bound for $d$ is approximately $d \geq 2n-1$. However, the best current bound is for $d$ greater than about $(en)^{2n+2}$ \cite{D2}. A related conjecture is the Green-Griffiths-Lang Conjecture.

\begin{conjecture}[Green-Griffiths-Lang Conjecture]
\label{conj-GGL}
If $X$ is a variety of general type, then there is a proper subvariety $Y \subset X$ containing all the entire curves of $X$.
\end{conjecture}

The Green-Griffiths-Lang Conjecture says that holomorphic images of $\CC$ under nonconstant maps do not pass through a general point of $X$. Conjecture \ref{conj-GGL} is well-studied for general hypersurfaces, and it is a natural result to prove on the way to proving Conjecture \ref{conj-Kobayashi}. We provide a new proof of the Kobayashi Conjecture using previous results on the Green-Griffiths-Lang Conjecture.

\begin{theorem}
A general hypersurface in $\PP^n$ of degree $d$ admits no nonconstant holomorphic maps from $\CC$ for $d \geq d_{2n-3}$, where $d_2 = 286, d_3 = 7316$ and
\[ d_n = \left\lfloor \frac{n^4}{3} (n \log(n \log(24n)))^n \right\rfloor .\]
\end{theorem}

Our proof appears to be substantially simpler than the previous proofs (compare with \cite{Siu, Deng, Brotbeck, D2}), and can be adapted in a straightforward way as others use jet bundles to obtain better bounds for Conjecture \ref{conj-GGL}. Unfortunately, the bound of about $(2n \log(n\log(n)))^{2n+1}$ that we obtain is slightly worse than Demailly's bound of $(en)^{2n+2}$. However, assuming the optimal result on the Green-Griffiths-Lang Conjecture, our technique allows us to prove the conjectured bound of $d \geq 2n-1$ for the Kobayashi Conjecture.

The Kobayashi Conjecture for complements of hypersurfaces has also been studied by several authors \cite{Dar, BD}. Using results of Darondeau \cite{Dar} along with our Grassmannian technique, we are able to prove the Kobayashi Conjecture for complements as well.

\begin{theorem}
If $X$ is a general hypersurface in $\PP^n$ of degree at least $d_{2n}$, where $d_n = (5n)^2 n^n$, then $\PP^n \setminus X$ is Brody hyperbolic.
\end{theorem}

Our bound of about $100 \cdot 2^n n^{2n+2}$ is slightly worse than Brotbeck and Deng's bound of about $e^3 n^{2n+6}$ \cite{BD}, but our proof is substantially shorter.


Our second application concerns the Chow equivalence of points on very general complete intersections. Chen, Lewis, and Sheng \cite{CLS} make the following conjecture, which is inspired by work of Voisin \cite{voisinChow, V2,V3}. 

\begin{conjecture}
\label{conj-CLS}
Let $X \subset \PP^n$ be a very general complete intersection of multidegree $(d_1, \dots, d_k)$. Then for every $p \in X$, the space of points of $X$ rationally equivalent to $p$ has dimension at most $2n-k-\sum_{i=1}^k d_i$. If $2n-k-\sum_{i=1}^k d_i < 0$, we understand this to mean that $p$ is equivalent to no other points of $X$.
\end{conjecture}

If this Conjecture holds, then the result is sharp \cite{CLS}. Voisin \cite{voisinChow, V2, V3} proves Conjecture \ref{conj-CLS} for hypersurfaces in the case $2n-d-1 < -1$. Chen, Lewis, and Sheng \cite{CLS} extend the result to $2n-d-1 = -1$, and also prove the analog of Voisin's bound for complete intersections. Both papers use fairly involved Hodge theory arguments. Roitman \cite{R1,R2} proves the $2n-k-\sum_{i=1}^k d_i = n-2$ case. Using Roitman's result, we prove all but the $2n-k-\sum_{i=1}^k d_i = -1$ case of Conjecture \ref{conj-CLS}, and in this case we prove the result holds with the exception of possibly countably many points. 

\begin{theorem}
\label{thm-chow}
If $X \subset \PP^n$ is a very general complete intersection of multidegree $(d_1, \dots, d_k)$, then no two points of $X$ are rationally Chow equivalent if $2n-k-\sum_{i=1}^k d_i < -1$. If $2n-k-\sum_{i=1}^k d_i = -1$, then the set of points rationally equivalent to another point of $X$ is a countable union of points. If $2n-k-\sum_{i=1}^k d_i \geq 0$, then the space of points of $X$ rationally equivalent to a fixed point $p \in X$ has dimension at most $2n-k-\sum_{i=1}^k d_i$ in $X$.
\end{theorem}

Together with Chen, Lewis and Sheng's result, this completely resolves Conjecture \ref{conj-CLS} in the case of hypersurfaces. Our method appears substantially simpler than the previous work of Voisin \cite{voisinChow, V2, V3} and Chen, Lewis, and Sheng \cite{CLS}, although in the case of hypersurfaces, we do not recover the full strength of Chen, Lewis, and Sheng's result.


The third result relates to Seshadri constants. Let $\epsilon(p,X)$ be the Seshadri constant of $X$ at the point $p$, defined to be the infimum of $\frac{\deg C}{\mult_p C}$ over all curves $C$ in $X$ passing through $p$. Let $\epsilon(X)$ be the Seshadri constant of $X$, defined to be the infimum of the $\epsilon(p,X)$ as $p$ varies over the hypersurface.

\begin{theorem}
\label{thm-seshadri}
Let $r > 0$ be a real number. If for a very general hypersurface $X_0 \subset \PP^{2n-1}$ of degree $d$ the Seshadri constant $\epsilon(p,X_0)$ of $X_0$ at a general point $p$ is at least $r$, then for a very general $X \subset \PP^n$ of degree $d$, the Seshadri constant $\epsilon(X)$ of $X$ is at least $r$.
\end{theorem}

The layout of the paper is as follows. In Section \ref{sec-technique}, we lay out our general technique, and immediately use it to prove Theorem \ref{thm-seshadri}. In Section \ref{sec-hyperbolicity}, we discuss how to use the results of Section \ref{sec-technique} to prove hyperbolicity results. In Section \ref{sec-cycles}, we discuss how to prove Theorem \ref{thm-chow}.

\subsection*{Acknowledgements}
We would like to thank Xi Chen, Izzet Coskun, Jean-Pierre Demailly, Mihai P\u{a}un, Chris Skalit, and Matthew Woolf for helpful discussions and comments.

\section{The Technique}
\label{sec-technique}
We set some notation. Let $B$ be the moduli space of complete intersections of multidegree $(d_1, \dots, d_k)$ in $\PP^{n+k}$ and $\uU_{n,\bd} \subset \PP^{n+k} \times B$ be the variety of pairs $([p],[X])$ with $[X] \in B$ and $p \in X$. We refer to elements of $\uU_{n,\bd}$ as pointed complete intersections. When we talk about the codimension of a countable union of subvarieties of $\uU_{n,\bd}$, we mean the minimum of the codimensions of each component.

We need the following result from \cite{RY}.
\begin{proposition}
\label{GrassmanProp}
Let $C \subset \gG(r-1,m)$ be a nonempty family of $r-1$-planes of codimension $\epsilon > 0$, and let $B \subset \gG(r,m)$ be the space of $r$ planes that contain some $r-1$-plane $c$ with $[c] \in C$. Then $\codim(B \subset \gG(r,m)) \leq \epsilon -1$.
\end{proposition}
\begin{proof}
For the reader's convenience, we sketch the proof. Consider the incidence-correspondence $I = \{ ([b],[c]) | \: [b] \in B, [c] \in C \} \subset \gG(r,m) \times \gG(r-1,m)$. The fibers of $\pi_2$ over $C \subset \gG(r-1,m)$ are all $\PP^{m-r}$'s, while for a general $[b] \in B$, the fiber $\pi_1^{-1}([b])$ has codimension at least $1$ in the $\PP^r$ of $r-1$-planes contained in $b$ (since otherwise it can be shown that $C = \gG(r,m)$). The result follows by a dimension count.
\end{proof}

We need a few other notions for the proof. A parameterized $r$-plane in $\PP^m$ is a degree one injective map $\Lambda: \PP^r \to \PP^m$. Let $G_{r,m,p}$ be the space of parameterized $r$-planes in $\PP^m$ whose images pass through $p$. If $(p,X)$ is a pointed hypersurface in $\PP^m$, a parameterized $r$-plane section of $(p,X)$ is a pair $(\Lambda^{-1}(p), \Lambda^{-1}(X)) =: \phi^* (p,X)$, where $\Lambda: \PP^r \to \PP^m$ is a parameterized $r$-plane whose image does not lie entirely in $X$. We say that $\Lambda: \PP^r \to \PP^m$ contains $\Lambda': \PP^{r-1} \to \PP^m$ if $\Lambda(\PP^r)$ contains $\Lambda'(\PP^{r-1})$.

\begin{corollary}
\label{GrassmanCor}
If $C \subset G_{r-1,m,p}$ is a nonempty subvariety of codimension $\epsilon > 0$ and $B \subset G_{r,m,p}$ is the subvariety of parameterized $r$-planes that contain some $r-1$-plane $[c] \in C$, then $\codim(B \subset G_{r,m,p}) \leq \epsilon -1$.
\end{corollary}

Let $\xX_{n,\bd} \subset \uU_{n,\bd}$ be an open subset. For instance, $\xX_{n,\bd}$ might be equal to $\uU_{n,\bd}$ or it might be the universal complete intersection over the space of smooth complete intersections. Our main technical tool is the following.

\begin{theorem}
\label{thm-technicaltool}
Suppose we have an integer $m$ and for each $n \leq n_0$ we have $Z_{n,d} \subset \xX_{n,\bd}$ a countable union of locally closed varieties satisfying: \begin{enumerate}
\item If $(p,X) \in Z_{n,\bd}$ and is a parameterized hyperplane section of $(p',X')$, then $(p',X') \in Z_{n+1,\bd}$.
\item $Z_{m-1,\bd}$ has codimension at least $1$ in $\xX_{m-1,d}$.
\end{enumerate}
Then the codimension of $Z_{m-c,\bd}$ in $\xX_{m-c,\bd}$ is at least $c$.
\end{theorem}
\begin{proof}[Proof of Theorem \ref{thm-technicaltool}]
We adopt the method from \cite{RY}. We prove that for a very general point $(p_0,X_0)$ of a component of $Z_{m-c,\bd}$, there is a variety $\FF_{m-c}$ and a map $\phi: \FF_{m-c} \to \uU_{m-c,\bd}$ with $(p_0,X_0) \in \phi(\FF_{m-c})$ and $\codim(\phi^{-1}(Z_{m-c,\bd}) \subset \FF_{m-c}) \geq c$. This suffices to prove the result.

So, let $(p_0, X_0)$ be a general point of a component of $Z_{m-c,\bd}$, and let $(p_1,X_1) \in \xX_{m-1,\bd}$ be very general, so that $(p_1,X_1)$ is not in the closure of any component of $Z_{m-1,\bd}$ by hypothesis 2. Choose $(p,Y) \in \xX_{N,\bd}$ for some sufficiently large $N$ such that $(p_0,X_0)$ and $(p_1,X_1)$ are parameterized linear sections of $(p,Y)$. Then for all $n < N$, let $\FF_{n}$ be the space of parameterized $n$-planes in $\PP^N$ passing through $p$ such that for $\Lambda \in \FF_{n}$, $\Lambda^*(p,Y)$ is in $\xX_{n,\bd}$. This means that $\FF_n$ is an open subset of $G_{n,N,p}$. Let $\phi: \FF_{n} \to \xX_{n,\bd}$ be the map sending $\Lambda: \PP^n \to \PP^N$ to $\Lambda^*(p,Y)$. 

We prove that $\codim(\phi^{-1}(Z_{m-c,\bd}) \subset \FF_{m-c}) \geq c$ by induction on $c$. For the $c=1$ case, we see by construction that $\phi^{-1}(Z_{m-1,\bd})$ has codimension at least $1$ in $\FF_{m-1}$, since $(p_1,X_1)$ is a parameterized $m-1$-plane section of $(p,Y)$ but is not in the closure of any component of $Z_{m-1,\bd}$. Now suppose we know that $\codim(\phi^{-1}(Z_{m-c,\bd}) \subset \FF_{m-c}) \geq c$. We use Corollary \ref{GrassmanCor} with $C$ equal to $\phi^{-1}(Z_{m-c-1,\bd})$. By hypothesis 1, we see that $B$ is contained in $\phi^{-1}(Z_{m-c,\bd})$. It follows from this that 
\[ c \leq \codim( \phi^{-1}(Z_{m-c,\bd}) \subset \FF_{m-c})) \leq \codim(B \subset \FF_{m-c}) \] 
\[ \leq \codim(C \subset \FF_{m-c-1}) -1 .  \]
Rearranging, we see that
\[ \codim( \phi^{-1}(Z_{m-c-1,\bd}) \subset \FF_{m-c-1}) = \codim(C \subset \FF_{m-c-1}) \geq c+1 . \]
The result follows.
\end{proof}

As an immediate application, we prove Theorem \ref{thm-seshadri}.

\begin{proof}[Proof of Theorem \ref{thm-seshadri}]
Let $r$ be given. Let $Z_{m,d} \subset \uU_{m,d}$ be the set of pairs $(p,X)$ where $\epsilon(p,X) < r$. We apply Theorem \ref{thm-technicaltool} to $Z_{m,d}$. We see that $Z_{m,d}$ is a countable union of algebraic varieties, and by hypothesis, $Z_{2n-1,d} \subset \uU_{2n-1,d}$ has codimension at least $1$. Now suppose that $(p_0,X_0) \in Z_{m,d}$. Then there is some curve $C$ in $X_0$ with $\frac{\deg C}{\mult_{p_0} C} < r$. Thus, for any $X$ containing $p$, we see that the Seshadri constant of $X$ at $p$ is at most $\frac{\deg C}{\mult_{p_0} C} < r$. This shows that the $Z_{m,d}$ satisfy the conditions of Theorem \ref{thm-technicaltool}, which shows that $Z_{n,d} \subset \uU_{n,d}$ has codimension at least $n$. By dimension reasons, this means that $\uU_{n,d}$ cannot dominate the space of hypersurfaces, so the result follows.
\end{proof}

\section{Hyperbolicity}
\label{sec-hyperbolicity}
Let $\xX_{n,d}$ be the universal hypersurface in $\PP^n$ over the open subset $U$ in the moduli space of all degree $d$ hypersurfaces in $\PP^n$ consisting of all smooth hypersurfaces. Many people have developed a technique for restricting the entire curves contained in a fiber of the map $\pi_2: \xX_{n,d} \to U$. See the article of Demailly for a detailed description of some of this work \cite{D2}. For a variety $X$, let $\ev: J_k(X) \to X$ be the space of $k$-jets of $X$. Then, if $X \subset \PP^n$ is a smooth degree $d$ hypersurface, there is a vector bundle $E_{k,m}^{GG} T_X^*$ whose sections act on $J_k(X)$. Global sections of $E_{k,m}^{GG} T_X^* \otimes \OO(-H)$ vanish on the $k$-jets of entire curves. This means that sections of $E_{k,m}^{GG} T_X^* \otimes \OO(-H)$ cut out a closed subvariety $S'_{k,m}(X) \subset J_k(X)$ such that any entire curve is contained in $\ev(S'_{k,m}(X))$. In fact, it can be shown that any entire curve is contained in the closure of $\ev(S_{k,m}(X))$, where $S_{k,m}(X) \subset J_k(X)$ is $S'_{k,m}(X)$ minus the space of singular $k$-jets.

The construction is functorial. In particular, if $V$ is the relative tangent bundle of the map $\pi_2$, there is a vector bundle $E_{k,m}^{GG} V^*$ whose restriction to each fiber of $\pi_2$ is $E_{k,m}^{GG} T_{X}^*$. Let $\yY_{n,d} \subset \xX_{n,d}$ be the locus of $(p,X) \in \xX_{n,d}$ such that $p \in \ev(S_{k,m}(X))$. Then by functoriality, $\yY_{n,d}$ is a finite union of locally closed varieties.

\begin{theorem}
\label{thm-Hyperbol}
Suppose that $\yY_{r-1,d} \subset \xX_{r-1,d}$ is codimension at least 1. Then $\yY_{r-c,d} \subset \xX_{r-c,d}$ is codimension at least $c$. In particular, if $\yY_{2n-3, d}$ is codimension at least $1$ in $\xX_{2n-3,d}$ and $d \geq 2n-1$, then a very general $X \subset \PP^n$ of degree $d_n$ is hyperbolic.
\end{theorem}
\begin{proof}
 We check that $\yY_{r-1,d}$ satisfies both conditions of Theorem \ref{thm-technicaltool}. Condition 2 is a hypothesis. Condition 1 follows by the functoriality of Demailly's construction. Namely, if $(p,X_0)$ is a parameterized linear section of $(p,X)$, then the natural map $X_0 \to X$ induces a pullback map on sections 
 \[ H^0(E_{k,m}^{GG} T_X^* \otimes \OO(-H)) \to H^0(E_{k,m}^{GG} T_{X_0}^* \otimes \OO(-H)) , \]
compatible with the natural inclusion of jets $J_k(X_0) \to J_k(X)$. In particular, if some section $s$ of $H^0(E_{k,m}^{GG} T_X^* \otimes \OO(-H))$ takes a nonzero value on a jet $\alpha(j)$, where $j \in J_k(X_0)$ and $\alpha:J_k(X_0) \to J_k(X)$ is the natural inclusion, then the restriction of $s$ to $X_0$ takes a nonzero value on the original jet $j \in J_k(X_0)$. Thus, if $X_0$ has a nonsingular $k$-jet at $p$ which is annihilated by every section in $H^0(E_{k,m}^{GG} T_{X_0}^* \otimes \OO(-H))$, $X$ has such a $k$-jet as well.

To see the second statement, observe that by Theorem \ref{thm-technicaltool}, $\yY_{n,d}$ has codimension in $\xX_{n,d}$ at least $2n-3-n+1 = n-2$. It follows that a general $X$ of degree $d$ in $\PP^n$ satisfies that the image of any entire curve is contained in an algebraic curve. Since $d \geq 2n-1$, by a theorem of Voisin \cite{V2, V3}, any algebraic curve in $X$ is of general type. The result follows.
\end{proof}

The current best bound for the Green-Griffiths-Lang Conjecture is from Demailly \cite{D1,D2}. The version we use comes out of Demailly's proof.

\begin{theorem}[\cite{D2}, Section 10]
If $k$ and $m$ are sufficiently large, we have that $\yY_{n,d} \subset \xX_{n,d}$ is codimension at least $1$ for $d \geq d_n$, where $d_2 = 286, d_3 = 7316$ and
\[ d_n = \left\lfloor \frac{n^4}{3} (n \log(n \log(24n)))^n \right\rfloor .\]
\end{theorem}

Using this bound, we obtain the following.

\begin{corollary}
The Kobayashi Conjecture holds hypersurfaces when $d \geq d_{2n-3}$, where $d_2 = 286, d_3 = 7316$ and
\[ d_n = \left\lfloor \frac{n^4}{3} (n \log(n \log(24n)))^n \right\rfloor .\]
\end{corollary}

This bound of about $(2n \log (n\log n))^{2n+1}$ is slightly worse than the best current bound for the Kobayashi Conjecture from \cite{D2}, which is about $(en)^{2n+2}$. However, our technique is strong enough to allow us to prove the optimal bound from the Kobayashi Conjecture, provided one could prove the optimal result for the Green-Griffiths-Lang Conjecture.

\begin{corollary}
If $\yY_{d-2,d}$ has codimension at least $1$ in $\xX_{d-2,d}$ (as we would expect from the Green-Griffiths-Lang Conjecture), then a very general hypersurface of degree $d \geq 2n-1$ in $\PP^n$ is hyperbolic.
\end{corollary}
\begin{proof}
We apply Theorem \ref{thm-Hyperbol}. We know that if $\yY_{2n-3, d}$ is codimension at least $1$ in $\xX_{2n-3,d}$, then the Kobayashi Conjecture holds for hypersurfaces in $\PP^n$ of degree $d$. We apply this result with $d = 2n-1$.
\end{proof}

Work has also been done on the hyperbolicity of complements of hypersurfaces in $\PP^n$. There are similar jet bundle techniques in this case. Given a variety $Z$, and ample line bundle $A$, a subsheaf $V \subset T_Z$ and a subvariety $X \subset Z$, one can construct the vector bundles $E_{k,m} V^*(\log X)$, sections of which act on $k$-jets of $Z \setminus X$. It can be shown that any section of $H^0(E_{k,m} V^*(\log X) \otimes A^{*})$ must vanish on the $k$-jet of any entire curve in $Z \setminus X$. Then sections of $H^0(E_{k,m} V^*(\log X) \otimes A^{*})$ cut out a subvariety $S_{k,m}$ in the locus of nonsingular $k$-jets on $Z \setminus X$ such that any entire curve lies in the closure of $\ev(S_{k,m})$. Darondeau \cite{Dar} studies these objects for hypersurfaces in $\PP^n$. Namely, he proves the following.

\begin{theorem}[Darondeau]
\label{thm-Darondeau}
Let $d \geq (5n)^2 n^n$, let $U_{n,d}$ be the space of smooth degree $d$ hypersurfaces in $\PP^n$ and consider the space $\PP^n \times U_{n,d}$, with divisor $\xX_{n,d}$ corresponding to the universal hypersurface. Let $V = T_{\pi_2}$ be the relative tangent space of projection of $\PP^n \times U_{n,d}$ onto $U_{n,d}$, and consider the locus $S_{k,m}$ cut out by sections in $H^0(E_{k,m} V^*(\log \xX_{n,d}) \otimes \OO(-H))$. Then $\ev(S_{k,m})$ is codimension at least $2$ in $\PP^n \times U_{n,d}$.
\end{theorem}

Using our technique, we obtain the following effective form of the Kobayashi Conjecture.

\begin{theorem}
Let $Z_{n,d} \subset \PP^n \times U_{n,d} \setminus \xX_{n,d}$ be the locus of pairs $(p,X)$ where $p \in (\PP^n \setminus X) \setminus \ev(S_{k,m})$. Then if $Z_{r-1,d}$ has codimension at least $1$ in $(\PP^{r-1} \times U_{r-1,d}) \setminus \xX_{r-1,d}$, we have that $Z_{r-c,d}$ has codimension at least $c$ in $(\PP^{r-c} \times U_{r-c,d}) \setminus \xX_{r-c,d}$.
\end{theorem}
\begin{proof}
The proof is very similar in spirit to the proof of Theorem \ref{thm-technicaltool}, but we give a new proof here for completeness. Let $(p, X_0)$ be a general point of a component of $Z_{r-c,d}$. We will find a family $\phi: \FF_{r-c} \to \PP^{r-c} \times U_{r-c,d} \setminus \xX_{r-c,d}$ with $\phi^{-1}(Z_{r-c})$ having codimension at least $c$ in $\FF_{r-c}$. 

Let $(p, X_1) \in \PP^{r-1} \times U_{r-1,d} \setminus \xX_{r-1,d}$ be very general so that $(p,X_1)$ is not in the closure of any component of $Z_{r-1,d}$. Let $(p,Y)$ be a hypersurface in some high dimensional projective space $\PP^N$ together with a point $p \in \PP^N \setminus Y$ such that $(p,X_0)$ and $(p,X_1)$ are both parameterized linear sections of $(p,Y)$. Let $\FF_{r-c}$ be the space of parameterized linear spaces in $\PP^N$ passing through $p$. Then we have a natural map $\phi: \FF_{r-c} \to \PP^{r-c} \times U_{r-c,d} \setminus \xX_{r-c,d}$ sending the parameterized $r-c$ plane $\Lambda$ to $\Lambda^* (p,Y)$. By construction, the image of $\phi$ will certainly contain the point $(p,X_0)$, and since $(p,X_1)$ is very general, $\phi^{-1}(Z_{r-1,d})$ will have codimension at least $1$ in $\FF_{r-1}$. Observe that if $(p,Y_0) \in Z_{r-c-1,d}$ is a linear section of $(p,Y_1)$, then $(p,Y_1) \in Z_{r-c,d}$, since if there is a nonsingular jet $j$ at $p$ such that all the sections of $E_{k,m} T_{\PP^{r-c-1}}^*(\log Y_0) \otimes \OO(-H)$ vanish on $j$, then certainly all the sections of $E_{k,m} T_{\PP^{r-c}}^*(\log Y_1) \otimes \OO(-H)$ will vanish on $j$. By repeated application of Corollary \ref{GrassmanCor}, we see that the codimension of $\phi^{-1}(Z_{r-c,d})$ in $\FF_{r-c,d}$ is at least $c$. This concludes the proof.
\end{proof}

It follows that if $Z_{n,d} \subset \PP^n \times U_{n,d} \setminus \xX_{n,d}$ has codimension at least $1$ for $d \geq d_n$, then a very general hypersurface complement in $\PP^n$ is Brody hyperbolic when $d \geq d_{2n-1}$. Using this together with Darondeau's result, we obtain the following corollary, weakening the bound a bit for brevity.

\begin{corollary}
Let $d \geq (10n)^2 (2n)^{2n}$. Then if $X \subset \PP^n$ is very general, the complement $\PP^n \setminus X$ is Brody hyperbolic.
\end{corollary}

\section{$0$-cycles}
\label{sec-cycles}
Let $R_{\PP^1, X, p} = \{ q \in X | Nq \sim Np \text{ for some integer $N$} \}$, where the relation $\sim$ means Chow equivalent. The goal of this section is to prove all but the $2n - \sum_{i=1}^k (d_i - 1) = -1$ case of the following conjecture of Chen, Lewis and Sheng \cite{CLS}.

\begin{conjecture}
\label{mainConj}
Let $X \subset \PP^{n}$ be a very general complete intersection of multidegree $(d_1, \dots, d_k)$. Then for every $p \in X$, $\dim R_{\PP^1, X, p} \leq 2n -k - \sum_{i=1}^k d_i$. 
\end{conjecture}

Here, we adopt the convention that $\dim R_{\PP^1, X, p}$ is negative if $R_{\PP^1, X, p} = \{p \}$. Together with the main result of \cite{CLS}, this completely resolves Conjecture \ref{mainConj} in the case of hypersurfaces.

Chen, Lewis and Sheng consider the more general notion of $\Gamma$ equivalence, although we are unable to prove the $\Gamma$ equivalence version here. The special case $\sum_i d_i = n+1$ is a theorem of Roitman \cite{R1, R2} and the case $2n -k - \sum_{i=1}^k d_i \leq -2$ is a theorem of Chen, Lewis and Sheng \cite{CLS} building on work of Voisin \cite{voisinChow, V2, V3}, who proves the result only for hypersurfaces. Chen, Lewis and Sheng prove Conjecture \ref{mainConj} for hypersurfaces and for arbitrary $\Gamma$ in the boundary case $2n - k - \sum_{i=1}^k d_i = -1$ in \cite{CLS}. The case $2n - k - \sum_{i=1}^k d_i = -1$ appears to be the most difficult, and is the only one we cannot completely resolve with our technique.

We provide an independent proof of all but the $2n - k - \sum_{i=1}^k d_i = -1$ case of Conjecture \ref{mainConj}. Aside from Roitman's result, this is the first result we are aware of addressing the case $2n - k - \sum_{i=1}^k d_i \geq 0$. We rely on the result of Roitman in our proof, but not the results of Voisin \cite{voisinChow, V2} or Chen, Lewis, and Sheng \cite{CLS}.

Let $E_{n,\bd} \subset \uU_{n,\bd}$ be the set of $(p,X)$ such that $R_{\PP^1, X, p}$ has dimension at least $1$. Let $G_{n,\bd} \subset \uU_{n,\bd}$ be the set of $(p,X)$ such that $R_{\PP^1,X,p}$ is not equal to $\{p\}$. Both $E_{n,\bd}$ and $G_{n,\bd}$ are countable unions of closed subvarieties of $\uU_{n,\bd}$. When we talk about the codimension of $E_{n,\bd}$ or $G_{n,\bd}$ in $\uU_{n,\bd}$, we mean the minimum of the codimensions of each component. We prove Conjecture \ref{mainConj} by proving the following theorem.

\begin{theorem}
\label{mainThm}
The codimension of $E_{n,\bd}$ in $\uU_{n,\bd}$ is at least $-n + \sum_i d_i$ and the codimension of $G_{n,\bd}$ in $\uU_{n,\bd}$ is at least $-n-1 + \sum_i d_i$.
\end{theorem}

\begin{corollary}
Conjecture \ref{mainConj} holds for $2n- k - \sum_{i=1}^k d_i \neq -1$. In the special case $2n-k- \sum_{i=1}^k d_i = -1$, the space of $p \in X$ Chow-equivalent to some other point of $X$ has dimension $0$ (i.e., is a countable union of points) but might not be empty as Conjecture \ref{mainConj} predicts.
\end{corollary}
\begin{proof}
First we consider the case $2n - k - \sum_i d_i \geq 0$. Let $\pi_1: \uU_{n,\bd} \to B$ be the projection map. If $\pi_1|_{E_{n,\bd}}$ is not dominant, then the result holds trivially. Thus, we may assume that the very general fiber of $\pi_1|_{E_{n,\bd}}$ has dimension $n - k - \codim(E_{n,\bd} \subset \uU_{n,\bd})$. If the bound on $E_{n,\bd}$ from Theorem \ref{mainThm} holds, then the space of points $p$ of $X$ with positive-dimensional $R_{\PP^1,X,p}$ has dimension at most $2n - k - \sum_i d_i$, which implies Conjecture \ref{mainConj} in the case $2n - k - \sum_i d_i \geq 0$. 

Now we consider the situation for $2n - k - \sum_i d_i \leq -1$. Conjecture \ref{mainConj} states that $\pi_1|_{G_{n,\bd}}$ is not dominant for this range. By Theorem \ref{mainThm}, we see that the dimension of $G_{n,\bd}$ is less than the dimension of $B$ if $2n - k - \sum_i d_i \leq -2$, proving Conjecture \ref{mainConj}. In the case $2n - k-\sum_i d_i = -1$, the dimension of $G_{n,\bd}$ is at most that of $B$, which shows that there are at most finitely many points of $X$ equivalent to another point of $X$. This proves the result. 
\end{proof} 

Our technique would prove Conjecture \ref{mainConj} in all cases if we knew that $G_{n,\bd}$ had codimension $-n + \sum_{i=1}^k d_i$ in $\uU_{n,\bd}$. However, this is not true for Calabi-Yau hypersurfaces. 

\begin{proposition}
A general point of a very general Calabi-Yau hypersurface $X$ is rationally equivalent to at least one other point of the hypersurface.
\end{proposition}
\begin{proof}
Let $X$ be a very general Calabi-Yau hypersurface. Then we claim that a general point of $X$ is Chow equivalent to another point of $X$. To see this, observe that any point $p$ of $X$ has finitely many lines meeting $X$ to order $d-1$ at $p$. Such a line meets $X$ in a single other point. Moreover, every point of $X$ has a line passing through it that meets $X$ at another point of $X$ with multiplicity $d-1$. Thus, let $q_1$ be a general point of $X$, let $\ell_1$ be a line through $p$ meeting $X$ at a second point, $p$ to order $d-1$, and $\ell_2$ be a different line meeting $X$ at $p$ to order $d-1$. Let $q_2$ be the residual intersections of $\ell_2$ with $X$. Then $q_1 \sim q_2$.
\end{proof}

\begin{proof}[Proof of Theorem \ref{mainThm}]
First consider the bound on $E_{n,\bd}$. By Roitman's Theorem plus the fact that $R_{\PP^1,X,p}$ is a countable union of closed varieties, we see that $E_{-1+\sum_{i=1}^k d_i,\bd}$ has codimension at least one in $\uU_{-1 + \sum_{i=1}^k d_i,\bd}$. We note that if $p \sim q$ as points of $Y$, and $Y \subset Y'$, then $p \sim q$ as points of $Y'$ as well. The rest of the result follows from Theorem \ref{thm-technicaltool} using $Z_{n,\bd} = E_{n,\bd}$.

Now consider $G_{n,\bd}$. From Roitman's theorem, it follows that a very general point of a Calabi-Yau complete intersection $X$ is equivalent to at most countably many other points of $X$. Thus, a very general hyperplane section of such an $X$ satisfies the property that the very general point is equivalent to no other points of $X$. From this, we see that $G_{-2+\sum_{i=1}^k d_i,\bd}$ has codimension at least $1$ in $\uU_{-2+\sum_{i=1}^k d_i,\bd}$. Together with Theorem \ref{thm-technicaltool}, this implies the result.
\end{proof}

\bibliographystyle{plain}

\end{document}